\documentclass[reqno,11pt]{amsart}

\usepackage{amssymb}
\usepackage{amsgen}
\usepackage{amsmath}
\usepackage{amsthm}
\usepackage{mathrsfs}
\usepackage{cite}
\usepackage{amsfonts}

\newcommand{\J}{{\mathrel{\mathscr J}}} 
\newcommand{\R}{{\mathrel{\mathscr R}}} 
\newcommand{\eL}{{\mathrel{\mathscr L}}} 

\newcommand{\inv}{^{-1}}

\newcommand{\ov}[1]{\ensuremath{\overline {#1}}}

\usepackage{xcolor}

\newtheorem{Thm}{Theorem}
\newtheorem{Prop}[Thm]{Proposition}

\newtheorem{Lemma}[Thm]{Lemma}
{\theoremstyle{definition}
}
{\theoremstyle{remark}
}
\newtheorem{Cor}[Thm]{Corollary}
{\theoremstyle{remark}
}
{\theoremstyle{remark}
}

{\theoremstyle{remark}
}
{\theoremstyle{remark}
}
{\theoremstyle{remark}
}

\title{Linear conjugacy}

\author{Benjamin Steinberg}
\address{%
    Department of Mathematics\\
    City College of New York\\
    Convent Avenue at 138th Street\\
    New York, New York 10031\\
    USA}
\email{bsteinberg@ccny.cuny.edu}

\thanks{This work was partially supported by a grant from the Simons Foundation(\#245268
to Benjamin Steinberg), the Binational Science Foundation of Israel and the US (\#2012080 to Benjamin Steinberg), by a CUNY Collaborative Incentive Research Grant, by  NSA MSP \#H98230-16-1-0047 and by a Fulbright Scholar award.}
\date{\today}


\begin{document}

\begin{abstract}
We say that two elements of a group or semigroup are $\Bbbk$-linear conjugates if their images under any linear representation over $\Bbbk$ are conjugate matrices.  In this paper we characterize $\Bbbk$-linear conjugacy for finite semigroups (and, in particular, for finite groups) over an arbitrary field $\Bbbk$.
\end{abstract}

\maketitle

\section{Introduction}
This article is motivated by a mathoverflow question asked by James Propp~\cite{Propp}.  A well-known lemma of Brauer~\cite{Kovacs.perm} asserts that two permutation matrices are similar if and only if the corresponding permutations are conjugate and the question was whether the same is true for matrices corresponding to functions.  The answer for functions is more complicated.

Let $S$ be a semigroup and $\Bbbk$ a field.   We say that $s,t\in S$ are \emph{$\Bbbk$-linear conjugates} if, for every linear representation $\rho\colon S\to M_n(\Bbbk)$, there is an invertible matrix $A\in GL_n(\Bbbk)$ such that $A\rho(s)A^{-1}=\rho(t)$.  This is clearly an equivalence relation on $S$.  Also note that if $s$ and $t$ are $\Bbbk$-linear conjugates, then so are $s^k$ and $t^k$ for all $k\geq 1$.  When $\Bbbk$ is the field of complex numbers, then we just say that $s,t$ are \emph{linear conjugates}.  Observe that if $\Bbbk$ is a subfield of $\mathbb F$, then $\mathbb F$-linear conjugates are also $\Bbbk$-linear conjugates.  This is a consequence of the fact that the rational canonical form of a matrix does not change when you extend the scalars.


For finite groups, linear conjugacy reduces to conjugacy.  Indeed, conjugate elements of any group are $\Bbbk$-linear
 conjugates over any field $\Bbbk$.  If $G$ is a finite group and $g,h\in G$ are linear conjugates, then every complex character of $G$ coincides on $g$ and $h$.  As the irreducible characters of $G$ form a basis for the space of functions constant on conjugacy classes, we deduce that $g,h$ are conjugate in $G$. For finite semigroups, the situation is a bit more complex, as we shall see. Nonetheless, there is a syntactic description of linear conjugacy for finite semigroups that seems to be interesting in its own right.  We give, in fact, a characterization of $\Bbbk$-linear conjugacy for finite semigroups over any field $\Bbbk$.

\section{Linear conjugacy for finite semigroups}
From now on, all semigroups are assumed finite.  A reference for semigroup representation theory is~\cite{repbook}.
Fix a semigroup $S$.  As usual, we shall denote by $s^{\omega}$ the idempotent power of $s\in S$ and put $s^{\omega+j}=s^js^{\omega}$ for $j\geq 1$; note that $(s^{\omega+1})^j=s^{\omega+j}$ for $j\geq 1$. If $|S|=m$, then $s^{\omega}=s^{m!}$.   Recall that $s,t\in S$ are \emph{generalized conjugates} if there exist $x,x'\in S$ such that $xx'x=x$, $x'xx'=x'$, $x'x=s^{\omega}$, $xx'=t^{\omega}$ and $xs^{\omega+1}x'=t^{\omega+1}$.  Note that this implies that $x't^{\omega+1}x=s^{\omega+1}$ and, in fact, generalized conjugacy is an equivalence relation.  It was proved independently by McAlister~\cite{McAlisterCharacter} and by Rhodes and Zalcstein~\cite{RhodesZalc} that $s,t$ are generalized conjugates if and only if $\chi(s)=\chi(t)$ for all complex characters $\chi$ of $S$.

Two elements $s,t\in S$ are $\J$-equivalent, written $s\J t$, if they generate the same principal two-sided ideal.  Similarly, they are $\mathscr L$-equivalent, written $s\eL t$, if the generate the same principal left ideal and they are $\mathscr R$-equivalent, written $s\R t$, if they generate the same principal right ideal.  We write $J_s$, $L_s$ and $R_s$ for the respective $\mathscr J$-, $\mathscr L$- and $\mathscr R$-classes of $s$.  In a finite semigroup, $sS^1\cap J_s=R_s$ and $S^1s\cap J_s=L_s$ (where $S^1$ is the result of adjoining an identity to $S$).

 Notice that if $|S|=n$, then $s^n\J s^k$ for all $k\geq n$.  Another classical fact that we shall need is that if $e,f\in S$ are idempotents and $x\eL e$ and $x\R f$, then there exists $x'\in S$ with $x'\R e$ and $x'\eL f$ such that $xx'x=x$, $x'xx'=x'$, $xx'=f$ and $x'x=e$.
 The \emph{maximal subgroup} $G_e$ at an idempotent $e\in S$ is the group of units of the monoid $eSe$ (with  identity $e$). In a finite semigroup $S$, one has that $G_e=eS\cap Se\cap J_e=eSe\cap J_e$ for an idempotent $e$.    The group $G_e$ acts freely on the right of $L_e$ by multiplication.  If $f$ is an idempotent in $J_e$, one has that $fS\cap L_e= R_f\cap L_e\neq \emptyset$.  This uses the stability of finite semigroups.
  See~\cite[Appendix~A]{qtheor} or~\cite{CP} for the necessary details on the algebraic theory of finite semigroups.

 The main goal of this article is to provide a syntactic description of $\Bbbk$-linear conjugacy for any field $\Bbbk$.   For example, linear conjugacy has the following syntactic formulation, to be proved shortly.

\begin{Thm}\label{t:main.thm}
Let $S$ be a finite semigroup.  Then $s,t\in S$ are linear conjugates if and only if
\begin{enumerate}
  \item $s^k\J t^k$ for all $k\geq 1$ (or, equivalently, for all $1\leq k\leq |S|$);
  \item $s$ and $t$ are generalized conjugates.
\end{enumerate}
\end{Thm}

Let $\Bbbk$ be a field. Then  $s,t\in S$ are said to be \emph{$\Bbbk$-character equivalent} if $\chi(s)=\chi(t)$ for each character $\chi$ of $S$ over $\Bbbk$.  Recall that the \emph{character} of a representation $\rho\colon S\to M_n(\Bbbk)$ is the mapping $\chi\colon S\to \Bbbk$ sending $s$ to the trace of $\rho(s)$.   For example, $s$ and $t$ are $\mathbb C$-character equivalent if and only if they are generalized conjugates.  Character equivalence over an arbitrary field was described in~\cite{ourcharacter}.  Let us now formulate our main result, which can then be made  explicit using the results of~\cite{ourcharacter}.  Notice that if $\Bbbk$ is a subfield of $\mathbb F$, then $\mathbb F$-character equivalence implies $\Bbbk$-character equivalence because every matrix representation over $\Bbbk$ is a representation over $\mathbb F$.

\begin{Thm}\label{t:main.thm2}
Let $S$ be a finite semigroup and $\Bbbk$ a field.  Then $s,t\in S$ are linear conjugates if and only if
\begin{enumerate}
  \item $s^k\J t^k$ for all $k\geq 1$ (or, equivalently, for all $1\leq k\leq |S|$);
  \item $s$ and $t$ are $\mathbb Q$-character equivalent;
  \item $s$ and $t$ are $\Bbbk$-character equivalent.
\end{enumerate}
\end{Thm}

Note that if $\Bbbk$ has characteristic zero or characteristic relatively prime to the order of each maximal subgroup of $S$, then the third item implies the second, as is easily seen from the description of $\Bbbk$-character equivalence given below.  We remark that Theorem~\ref{t:main.thm2} seems to be new for finite groups.

To describe the results of~\cite{ourcharacter}, we shall need further notation. If $p>0$ is a prime and $G$ is a finite group, then an element $g\in G$ is called \emph{$p$-regular} if it has order prime to $p$.    We shall consider all elements to be $p$-regular when $p=0$. So from now on let $p$ be $0$ or a prime number.  Each element $g\in G$ has a unique factorization $g=g(p)g(p')$ such that $g(p)g(p')=g(p')g(p)$, $g(p)$ has order a $p$-power and $g(p')$ is $p$-regular.  If $p=0$,  then $g=g(p')$ and $g(p)=1$. Otherwise, write $|g|=p^kr$ with $\gcd(p,r)=1$.  Then $g(p)=g^{m}$ and $g(p')=g^n$ where $m,n>0$ satisfy
\begin{alignat*}{2}
m & \equiv 1\bmod p^k, &\qquad n &\equiv 0\bmod p^k\\
m &\equiv 0\bmod r,  &\qquad  n&\equiv 1\bmod r.
\end{alignat*}

Recall that $s\in S$ is a \emph{group element} if $s$ generates a cyclic group, that is, $s=s^{\omega+1}$. One can then talk about $p$-regular group elements of $S$.  We put $s(p)=s^{\omega+1}(p)$ and $s(p')=s^{\omega+1}(p')$;  these are group elements.  If $p=0$, then $s(p)=s^{\omega}$ and $s(p')=s^{\omega+1}$.

Fix an algebraic closure $\ov \Bbbk$ of $\Bbbk$ and let $\xi$ be a primitive $n^{th}$-root of unity in $\ov{\Bbbk}$ where $n$ is the least common multiple of the orders of the $p$-regular group elements of $S$; note that $\gcd(n,p)=1$ if $p>0$. The Galois group $\mathrm{Gal}(\Bbbk(\xi)/\Bbbk)$ can be identified with a subgroup $H$ of $\mathbb Z_n^{\times}$ via the map $\sigma\mapsto k$ where $\sigma(\xi)=\xi^k$. For example, if $\Bbbk=\mathbb Q$, then $H=\mathbb Z_n^{\times}$. With this notation, $s,t\in S$ are $\Bbbk$-character equivalent if and only if there exist $x,x'\in S$ with $xx'x=x$, $x'xx'=x'$, $x'x=s^{\omega}$, $xx'=t^{\omega}$ and $xs(p')x'=t(p')^j$ with $j\in H$. See~\cite[Theorem~2.12]{ourcharacter}, where the result is stated for monoids but works equally well for semigroups.

For example,  $s,t$ are $\mathbb Q$-character equivalent if and only if there exist  $x,x'\in S$ with $xx'x=x$, $x'xx'=x'$, $x'x=s^{\omega}$, $xx'=t^{\omega}$ and $x\langle s^{\omega+1}\rangle x'=\langle t^{\omega+1}\rangle$ using that $\mathrm{Gal}(\Bbbk(\xi)/\Bbbk)=\mathbb Z_n^{\times}$.  Notice that $\mathbb C$-character equivalence (i.e., the relation of being generalized conjugates) implies $\Bbbk$-character equivalence for every field $\Bbbk$.

The proof of Theorem~\ref{t:main.thm2} consists of two steps: proving the necessity and the sufficiency of these conditions. Our proof of sufficiency uses the Fitting decomposition of a linear operator.

  Let $T$ be a linear operator on a finite dimensional $\Bbbk$-vector space $V$.
Note that there are chains of $T$-invariant subspaces
\begin{gather*}
\ker T\subseteq \ker T^2\subseteq\cdots\\
TV\supseteq T^2V\supseteq\cdots
\end{gather*}
and as soon as soon as two consecutive values of either of these chains are the same, the respective chain stabilizes. By finite dimensionality, each of these chains does stabilize.
For convenience, we put $\ker^{\infty} T=\bigcup_{k\geq 1}\ker T^k$ and $\mathop{\mathrm{im}^{\infty}} T=\bigcap_{k\geq 1}T^kV$ and call the latter the \emph{eventual range} of $T$. The following theorem is standard linear algebra; see~\cite[Theorem~5.38]{repbook}.

\begin{Thm}[Fitting decomposition]
Let $T$ be a linear operator on a finite dimensional $\Bbbk$-vector space $V$.  Then there is a unique direct sum decomposition $V=U\oplus W$ into $T$-invariant subspaces such that $T|_U$ is nilpotent and $T|_W$ is invertible.  Moreover, if $m>0$ is such that $\ker T^m=\ker T^{m+1}$ and $T^mV=T^{m+1}V$, then $U=\ker^{\infty} T=\ker T^m$ and $W=\mathop{\mathrm{im}^{\infty}} T=T^mV$.
\end{Thm}

We give a characterization of conjugacy of linear operators, based on the Fitting decomposition, inspired by ideas of Kov\'acs~\cite{Kovacs}.  If $T$ is a linear operator on an $n$-dimensional vector space $V$, the \emph{rank sequence} of $T$ is
\[\vec r(T)=(\dim TV,\dim T^2V,\ldots).\]  Note that the rank sequence is a weakly decreasing sequence of non-negative integers bounded by $n$, which becomes constant (and equal to the dimension of the eventual range of $T$) as soon as two consecutive values are equal.  In particular, there are only finitely many rank sequences of operators on an $n$-dimensional vector space.  Also note that $\vec r(T)=\vec r(T')$ implies that the eventual ranges of $T$ and $T'$ have the same dimension.

\begin{Cor}\label{c:kovacs}
Let $T,T'$ be linear operators on a finite dimensional $\Bbbk$-vector space $V$.  Then $T,T'$ are conjugate if and only if $\vec r(T)=\vec r(T')$ and there is a vector space isomorphism $F\colon \mathop{\mathrm{im}^{\infty}} T\to \mathop{\mathrm{im}^{\infty}} T'$ such that $FTv=T'Fv$ for all $v\in \mathop{\mathrm{im}^{\infty}} T$.
\end{Cor}
\begin{proof}
Trivially, if $A$ is an invertible operator with $ATA^{-1}=T'$, then $\vec r(T)=\vec r(T')$.  Also, by the uniqueness in the Fitting decomposition, $A(\mathop{\mathrm{im}^{\infty}}T)=\mathop{\mathrm{im}^{\infty}} T'$. Clearly, if $v\in \mathop{\mathrm{im}^{\infty}} T$, then $ATv=T'Av$.  Thus the conditions are necessary.

For sufficiency, note that $\dim \mathop{\mathrm{im}^{\infty}} T=\dim \mathop{\mathrm{im}^{\infty}} T'$ because $\vec r(T)=\vec r(T')$.  In light of the Fitting decomposition and the existence of the isomorphism $F$, to show that the $\Bbbk[x]$-module corresponding to the action of $T$ on $V$ is isomorphic to the $\Bbbk[x]$-module corresponding to the action of $T'$ on $V$ it suffices to show that the nilpotent operators $N=T|_{\ker^{\infty} T}$ and $N'=T'|_{\ker^{\infty} T'}$ have the same Jordan canonical form (note that they are both operators on a space of the same dimension).  Notice that $\dim T^iV-\dim T^{i+1}V=\dim N^i(\ker^{\infty} T)-\dim N^{i+1}(\ker^{\infty}T)$ is the number of Jordan blocks of $N$ of degree greater than $i$ for all $i\geq 0$.  Thus $\vec r(T)$ determines the Jordan canonical form of $N$; similarly, $\vec r(T')$ determines the Jordan canonical form of $N'$ and so $\vec r(T)=\vec r(T')$ implies that $N,N'$ have the same Jordan canonical form.  This proves that $T,T'$ are conjugate.
\end{proof}

The Fitting decomposition for the image of an element under a representation of a finite semigroup is easy to describe.

\begin{Prop}\label{p:fittingdecomp.finite}
Let $\rho\colon S\to M_n(\Bbbk)$ be a representation of a finite semigroup and put $V=\Bbbk^n$ with its usual left $\Bbbk S$-module structure.  Then, for $s\in S$, the Fitting decomposition of $\rho(S)$ is given by $\ker^{\infty} \rho(s)=(1-s^{\omega})V$ and $\mathop{\mathrm{im}^{\infty}} \rho(s)= s^{\omega}V$.
\end{Prop}
\begin{proof}
Choose $m>0$ such that $\mathop{\mathrm{im}^{\infty}} \rho(s)=\rho(s)^mV=\rho(s)^{m+k}V$ and $\ker^{\infty}\rho(s)=\ker \rho(s)^m=\ker \rho(s)^{m+k}$ for all $k\geq 1$.  As $s^{\omega}=s^N$ for some $N>m$, we conclude that $\mathop{\mathrm{im}^{\infty}} \rho(s)=s^{\omega}V$ and $\ker^{\infty}\rho(s)=\ker \rho(s^{\omega}) = (1-s^{\omega})V$, where the last equality uses that $s^{\omega}$ is idempotent.
\end{proof}


The Fitting decomposition essentially  reduces the problem from semigroups to groups.


We shall need the following key lemma; see~\cite{Kovacs.perm} for a proof.
\begin{Lemma}[Brauer's lemma]\label{l:perm.matrix}
Let $P,Q\in GL_n(\Bbbk)$ be permutation matrices.  Then $P$ and $Q$ are conjugate in $GL_n(\Bbbk)$ if and only if they are conjugate in the symmetric group $S_n$ (viewed as the group of $n\times n$ permutation matrices).
\end{Lemma}

Our next goal is to understand how the Galois action affects conjugacy.

\begin{Prop}\label{p:galois.act}
Let $\Bbbk$ be a field of characteristic $p\geq 0$ and $\xi$ a primitive $n^{th}$-root of unity in an algebraic closure $\ov{\Bbbk}$ of $\Bbbk$ where $\gcd(n,p)=1$ in the case that $p>0$.  Suppose that $T$ is a linear operator on a finite dimensional $\Bbbk$-vector space $V$ satisfying $T^n=1$.  Then $T$ and $T^j$ are conjugate for any $j\in H$ where $H$ is the subgroup of $\mathbb Z_n^{\times}$ corresponding to $\mathrm{Gal}(\Bbbk(\xi)/\Bbbk)$ as above.
\end{Prop}
\begin{proof}
First note that since $j\in \mathbb Z_n^{\times}$ and $T$ has finite order dividing $n$, it follows that $\langle T\rangle=\langle T^j\rangle$ and hence $T$ and $T^j$ have the same invariant subspaces of $V$. Consequently, they have the same cyclic invariant subspaces.  As $V$ is a direct sum of cyclic invariant subspaces, we may assume without loss of generality that $V$ is a cyclic invariant subspace for both $T$ and $T^j$.   Moreover, since  the polynomial $x^n-1$, which splits into distinct linear factors over $\Bbbk(\xi)$ by hypothesis on $n$, vanishes on both $T$ and $T^j$ it follows  that the minimal polynomials of $p(x)$ and $q(x)$ of $T$ and $T^j$, respectively, both split into distinct linear factors over $\Bbbk(\xi)$.  To prove the proposition, it suffices to show that $p(x)=q(x)$.

Let $\lambda_1,\ldots, \lambda_r$ be the roots of $p(x)$ in $\Bbbk(\xi)$.  As $p(x)$ has no repeated roots and $V$ is cyclic, there is a basis of $\Bbbk(\xi)\otimes_{\Bbbk} V$ such that $T$ is a diagonal matrix  with diagonal entries $\lambda_1,\ldots, \lambda_r$  and hence $T^j$ is also diagonal in this basis with diagonal entries $\lambda_1^j,\ldots, \lambda_r^j$.  Thus $\{\lambda_1^j,\ldots,\lambda_r^j\}$ are the roots of $q(x)$ in $\Bbbk(\xi)$.  As the $\lambda_i$ are $n^{th}$-roots of unity, and hence powers of $\xi$,  if $\sigma\in \mathrm{Gal}(\Bbbk(\xi)/\Bbbk)$ is the element with $\sigma(\xi)=\xi^j$, then $\sigma(\lambda_i)=\lambda_i^j$.  It follows that the roots of $q(x)$ are $\sigma(\lambda_1),\ldots, \sigma(\lambda_r)$ and hence $p(x)=q(x)$ because $\mathrm{Gal}(\Bbbk(\xi)/\Bbbk)$ permutes the roots of $p(x)$ and both $p(x)$ and $q(x)$ have no repeated roots and split over $\Bbbk(\xi)$.
\end{proof}

We are now prepared to prove our main result; Theorem~\ref{t:main.thm} is the special case that $\Bbbk=\mathbb C$.

\begin{proof}[Proof of Theorem~\ref{t:main.thm2}]
Assume that $\Bbbk$ has characteristic $p\geq 0$.
We begin with the necessity of (1), (2) and (3).
If $s,t$ are $\Bbbk$-linear conjugates, then each character of $S$ over $\Bbbk$ agrees on $s$ and $t$ and so (3) holds.  Suppose that $s^k$ is not $\mathscr J$-equivalent to $t^k$ for some $k\geq 1$.  Without loss of generality, assume that the principal ideal generated by $s^k$ is not contained in the principal ideal $I$ generated by $t^k$.  Let $S^1$ be the result of adjoining an identity to $S$.  Then $V=\Bbbk S^1/\Bbbk I$ is a left $\Bbbk S$-module annihilated by $t^k$ but not by $s^k$ (as $s^k(1+\Bbbk I)\neq \Bbbk I)$).  Therefore, if $\rho$ is the representation afforded by $V$, then $\rho(s)$ is not conjugate to $\rho(t)$.  Thus (1) holds.

The  proof of (2) is a bit trickier.  Put $e=s^{\omega}$ and $f=t^{\omega}$ and note that $e\mathrel{\mathscr J} f$ by (1).  Put $g=s^{\omega+1}$ and $h=t^{\omega+1}$.  Notice that $g$ and $h$ are group elements.  By considering the action of $s,t$ on $\Bbbk S^1$, and using that they are $\Bbbk$-linear conjugate, we deduce that $\langle s\rangle\cong \langle t\rangle$ via an isomorphism taking $s$ to $t$ and hence $g$ and $h$ have the same order.

The group $C=\langle g\rangle$ acts freely on the right of the $\mathscr L$-class $L_e$ of $e$ and we denote the orbit of $x\in L_e$ by $xC$.  We can define a $\Bbbk S$-module structure on $\Bbbk [L_e/C]$ by
\[s\cdot xC = \begin{cases} sxC, & \text{if}\ sx\in L_e\\ 0, & \text{else}\end{cases}\] for $x\in L_e$ and $s\in S$. By Corollary~\ref{c:kovacs}, Proposition~\ref{p:fittingdecomp.finite} and our assumption that $s$ and $t$ are $\Bbbk$-linear conjugates, there must be a vector space isomorphism $T\colon e\Bbbk [L_e/C]\to f\Bbbk [L_e/C]$ intertwining the actions of $s$ and $t$.  However, $s$ acts the same as $g$ on $e\Bbbk [L_e/c]$ and $t$ acts the same as $h$ on $f\Bbbk [L_e/C]$ and so $T$ intertwines the actions of $g$ and $h$. Note that $eL_e\cap L_e=G_e$ is the maximal subgroup of $S$ at $e$ and so $e[\Bbbk L_e/C]\cong \Bbbk [G_e/C]$ is a permutation module for $C$.   Also, $fL_e\cap L_e = R_f\cap L_e$ is permuted by $\langle h\rangle$ on the left and $C$ on the right with commuting actions and so $f\Bbbk [L_e/C]$ is a permutation module for $\langle h\rangle$.  Now $g$ has a fixed point on $G_e/C$, namely the coset $C$.  It follows from Lemma~\ref{l:perm.matrix} that $h$ has a fixed point $xC$ with $x\in R_f\cap L_e$; so $hxC=xC$.  By basic semigroup theory, there is then an element $x'\in R_e\cap L_f$ with $xx'x=x$, $x'xx'=x'$, $x'x=e$ and $xx'=f$. One easily checks that $\psi\colon G_f\to G_e$ given by $\psi(z) = x'zx$ is a group isomorphism and so $x'hx$ is an element of $G_e$ of the same order as $h$, and hence of the same order as $g$.  But $x'hxC=x'xC=C$ and so $x'hx\in C=\langle g\rangle$.  Thus $x'\langle h\rangle x=\langle x'hx\rangle =\langle g\rangle$ as $x'hx$ and $g$ have the same order.  We conclude that $s$ and $t$ are $\mathbb Q$-character equivalent.

To prove sufficiency, assume that (1), (2) and (3) hold. Let $n$ be the least common multiple of the orders of the $p$-regular group elements of $S$ and let $H\leq \mathbb Z_n^{\times}$ be the subgroup associated to $\mathrm{Gal}(\Bbbk(\xi)/\Bbbk)$ where $\xi$ is a primitive $n^{th}$-root of unity in a fixed algebraic closure of $\Bbbk$.  Note that $n$ is not divisible by the characteristic of $\Bbbk$.

Let $\rho\colon S\to M_r(\Bbbk)$ be a representation.  Put $V=\Bbbk^r$ with its usual left $\Bbbk S$-module structure.  From (1), and the well-known fact that two matrices are $\mathscr J$-equivalent if and only if they have the same rank~\cite[Section~2.2, Exercise 6]{CP}, it follows that $\vec r(\rho(s))=\vec r(\rho(t))$.  Thus to prove that $\rho(s)$ and $\rho(t)$ are conjugate, it suffices by Corollary~\ref{c:kovacs} and Proposition~\ref{p:fittingdecomp.finite} to construct a linear isomorphism $F\colon s^{\omega}V\to t^{\omega}V$ such that $F(sv)=tF(v)$ for all $v\in s^{\omega}V$.

Since $s,t$ are $\mathbb Q$-character equivalent, we can choose $x,x'\in S$ with $xx'x=x$, $x'xx'=x'$ and $x'x=s^{\omega}$, $xx'=t^{\omega}$ such that $h=xs^{\omega+1}x'$ generates the same cyclic group as $g=t^{\omega+1}$.
We first define a linear isomorphism $F'\colon s^{\omega}V\to t^{\omega}V$ such that $F'(sv)=hF'(v)$ for all $v\in s^{\omega}V$.
  Define $F'$ by $F'(v)=xv$ for $v\in s^{\omega}V$; clearly $F'$ is linear.  First note that $xv=xx'xv=t^{\omega}xv$ and so $F'(v)\in t^{\omega}V$.  Also,
\[hF'(v)=hxv=xs^{\omega+1}x'xv=xss^{\omega}x'xv=xss^{\omega}v=xsv=F'(sv)\]
using that $x'x=s^{\omega}$ and $s^{\omega}v=v$.  Similarly, there is a linear mapping $G\colon t^{\omega}V\to s^{\omega}V$ defined by $G(w)=x'w$ since $s^{\omega}x'w=x'xx'w=x'w$.  We claim that these mappings are mutually inverse.  Indeed, $GF'(v)=x'xv=s^{\omega}v=v$ for $v\in s^{\omega}V$; similarly $F'G(w)=xx'w=t^{\omega}w=w$ for $w\in t^{\omega}V$.  This shows that $F'$ is a linear isomorphism intertwining the action of $s$ on $s^{\omega}V$ and $h$ on $t^{\omega}V$.  
It, therefore, suffices to show that there is an invertible operator on $W=t^{\omega}V$ conjugating $h|_W$ to  $t|_W$ (or equivalently $g|_W$).  Also note that by construction $h$ is a generalized conjugate of $s$ and hence $\Bbbk$-character equivalent to $s$, and thus to $t$ by (3).

Note that since $h$ and $g$ generate the same cyclic subgroup $C$, they have the same invariant subspaces on $W$.  Write $|C|=p^rm$, where $p^r$ is interpreted as $1$ if $p=0$, and  $\gcd(p,m)=1$ if $p>0$. Then $h(p')$ and $g(p')$ both have order $m$ and hence generate the same cyclic subgroup $C'$ of $C$.
Observe that $W=V_1\oplus W'$ where $V_1$ is the generalized eigenspace of $1$ for $g|_W$ (which is also the generalized eigenspace of $1$ for $h|_W$ as they are both powers of each other) and $W'$ is a semisimple $\Bbbk C$-module not containing the trivial representation (since
\[\Bbbk C\cong \Bbbk [z]/((z-1)^{p^r})\times \Bbbk [z]/\left(\frac{z^m-1}{z-1}\right)\] and $\gcd(m,p)=1$ if $p>0$, whence $z^m-1$ splits into distinct linear factors over $\ov\Bbbk$).

 Since $g$ and $h$ have the same invariant subspaces on $V_1$, the vector space $V_1$ is a direct sum of indecomposable invariant subspaces and each indecomposable invariant subspace is isomorphic to a Jordan block with eigenvalue $1$ for both $g$ and $h$, it follows that $g|_{V_1}$ and $h|_{V_1}$ have the same Jordan canonical form and hence there is an invertible operator on $V_1$ conjugating $h|_{V_1}$  to $g|_{V_1}$.  Note that $h(p)$ and $g(p)$ act trivially on any semisimple $\Bbbk C$-module (since $h(p)-1$ and $g(p)-1$ are nilpotent in the commutative algebra $\Bbbk C$) and so $h|_{W'}=h(p')|_{W'}$ and $g|_{W'}=g(p')|_{W'}$.  As $h(p')$ and $g(p')$ have order $m$ prime to $p$, the subgroup $C'$ they generate has a semisimple algebra over $\Bbbk$.   Since $W'$ contains no copy of the trivial $\Bbbk C$-module and $h(p')|_{W'}=h|_{W'}$ and $g(p')|_{W'}=g|_{W'}$, it follows that $W'$ is the sum of all non-trivial isotypic components of $W$ for $C'$ and $V_1$ is the isotopic component of the trivial representation.  Let $\psi$ be the automorphism of $C'$ taking $h(p')$ to $g(p')$.   For $U$ a $\Bbbk C'$-module, let $U^{\psi}$ denote the $\Bbbk C'$-module with underlying vector space $U$ and module action $x\cdot u=\psi(x)u$ for $x\in C'$ and $u\in U$. If $Th(p')|_WT\inv = g(p')|_W=\psi(h(p'))|_W$ with $T\in GL(W)$, then $T$ provides an isomorphism $W\to W^{\psi}$.    It follows that if $\gamma$ is an irreducible representation of $C'$ then $T$ takes the isotypic component of $\gamma$ in $W$ to the isotypic component of $\gamma$ in $W^{\psi}$, which as a subspace of $W$ is the isotypic component of $\gamma\circ \psi\inv$ with respect to the original module structure.  Therefore,
 $T(V_1)=V_1$ and $T(W')=W'$.  Thus to get that $h|_{W'}=h(p')|_{W'}$ is conjugate to $g|_{W'}=g(p')|_{W'}$, it suffices to prove that $h(p')|_W$ is conjugate to $g(p')|_W$ as operators on $W$.

Since $h$ is $\Bbbk$-character equivalent to $t$, we can find $y,y'\in S$ with $yy'y=y$, $y'yy'=y'$, $yy'=t^{\omega}=h^{\omega}=y'y$ and $yh(p')y' = g(p')^j$ with $j\in H$.  Then $y,y'\in G_{t^{\omega}}$ and $y=y\inv$, and so $h(p')$ is conjugate to $g(p')^j$ in $G_{t^{\omega}}$ and hence they have conjugate actions on $W$.  Thus it suffices to show that $g(p')^j|_W$  is conjugate to $g(p')|_W$.
  Note that $g(p')$ is a $p$-regular group element of $S$ and hence has order dividing $n$.  Thus $g(p')|_{W}$ is conjugate to $g(p')^j|_{W}$ by Proposition~\ref{p:galois.act}.  This completes the proof.
\end{proof}

\section{Examples}
In this section we explore linear conjugacy in some important families of semigroups.

\subsection{Full transformation monoids}
Consider $T_n$, the full transformation monoid of degree $n$. Define the \emph{rank} of $f\in T_n$ to be the cardinality of its image.  It is well known that $f\J g$ if and only if they have the same rank~\cite[Theorem~2.9]{CP}.   An element $f\in T_n$ acts on the image of $f^{\omega}$ as a permutation.  One has that $f,g\in T_n$ are generalized conjugates if and only if $f^{\omega}$ and $g^{\omega}$ have the same rank and $f$ and $g$ have the same cycle structure as permutations of $\mathrm{Im}\ f^{\omega}$ and $\mathrm{Im}\ g^{\omega}$, respectively; see~\cite[Execise~7.10]{repbook}.  Two functions $f,g$ are conjugate by an element of $S_n$ if and only if they have isomorphic functional digraphs, where the functional digraph of $h\in T_n$ has vertex set $\{ 1,\ldots,n\}$ and an edge from $i$ to $h(i)$ for each $i\in \{1,\ldots, n\}$.

By the standard representation of $T_n$, we mean the representation $\rho\colon T_n\to M_n(\mathbb C)$ given by
\[\rho(f)_{ij} = \begin{cases}1, & \text{if}\ f(j)=i\\ 0, & \text{else.}\end{cases} \]

\begin{Thm}
Let $f,g\in T_n$.  Then the following are equivalent.
\begin{enumerate}
  \item $\mathrm{rank}(f^i)=\mathrm{rank}(g^i)$ for $i=1,\ldots, n$ and $f|_{\mathrm{Im}\ f^{\omega}}$ has the same cycle structure as $g|_{\mathrm{Im}\ g^{\omega}}$.
  \item $\rho(f)$ is similar to $\rho(g)$.
  \item $f$ and $g$ are linear conjugates.
\end{enumerate}
\end{Thm}
\begin{proof}
Theorem~\ref{t:main.thm} shows that (1) implies (3).  Clearly, (3) implies (2). Since the rank of a mapping $h$ is the same as the rank of the matrix $\rho(h)$, if $\rho(f)$ is similar to $\rho(g)$, then $\mathrm{rank}(f^i)=\mathrm{rank}(g^i)$ for $i=1,\ldots, n$.  Notice that the matrix of $\rho(f)|_{\mathrm{Im}\ \rho(f^{\omega})}$ is the permutation matrix for the action of $f$ on $\mathrm{Im}\ f^{\omega}$, and similarly for $g$.  If $\rho(f)$ is similar to $\rho(g)$, then these two permutation matrices are similar by Corollary~\ref{c:kovacs} and Proposition~\ref{p:fittingdecomp.finite}.  So by Lemma~\ref{l:perm.matrix} we deduce that  $f|_{\mathrm{Im}\ f^{\omega}}$ has the same cycle structure as $g|_{\mathrm{Im}\ g^{\omega}}$.  This completes the proof.
\end{proof}

Note that linear conjugacy in $T_n$ is strictly between generalized conjugacy and conjugacy by an element of $S_n$. Condition (2) was the subject of  James Propp's mathoverflow question~\cite{Propp} that prompted this work.

\subsection{Symmetric inverse monoids}
The symmetric inverse monoid $I_n$ (also called the rook monoid~\cite{Solomonrook}) is the monoid of all partial injective mappings on $\{1,\ldots, n\}$.  The \emph{rank} of a partial injection is the size of its image (or, equivalently, domain).  The group of units of $I_n$ is the symmetric group $S_n$.  It is well known that two elements of $I_n$ are $\mathscr J$-equivalent if and only if they have the same rank.  Also, if $f\in I_n$, then $f$ acts as a permutation of $\mathrm{Im}\ f^{\omega}$ and it is well known that $f,g\in I_n$ are generalized conjugates if and only if $f^{\omega}$ and $g^{\omega}$ have the same rank and $f|_{\mathrm{Im}\ f^{\omega}}$ has the same cycle structure as $g|_{\mathrm{Im}\ g^{\omega}}$; see~\cite[Exercise~7.8]{repbook}.

\begin{Thm}
Two elements of $I_n$ are linear conjugates if and only if they are conjugate by an element of $S_n$.
\end{Thm}
\begin{proof}
Clearly, if $f,g$ are conjugate by an element of $S_n$, then they are  linear conjugates.  If $f,g$ are linear conjugates, then by Theorem~\ref{t:main.thm} we have that $\mathrm{rank}(f^i)=\mathrm{rank}(g^i)$ for all  $i\geq 1$.  We also have that $f,g$ are generalized conjugates, which means that  $f|_{\mathrm{Im}\ f^{\omega}}$ has the same cycle structure as $g|_{\mathrm{Im}\ g^{\omega}}$.  It then follows from~\cite[Theorem~3.19]{repbook}  that $f,g$ are conjugate by an element of $S_n$.
\end{proof}

\subsection{Full matrix monoids}
Next we consider the monoid $M_n(\mathbb F_q)$ of $n\times n$ matrices over the field of $q$ elements $\mathbb F_q$.

\begin{Thm}
Let $q$ be a prime power.  Then $A,B\in M_n(\mathbb F_q)$ are linear conjugates if and only if they are similar matrices, that is, conjugate by an element of $GL_n(\mathbb F_q)$.
\end{Thm}
\begin{proof}
Clearly, if $A,B$ are similar, then they are linear conjugates.  On the other hand, if $A,B$ are linear conjugates, then since $\mathbb C$-character equivalence implies $\Bbbk$-character equivalence for any field $\Bbbk$,  (1) and (2) of Theorem~\ref{t:main.thm} are sufficient to guarantee $\mathbb F_q$-linear conjugacy by Theorem~\ref{t:main.thm2}.  Since the identity map is a representation of $M_n(\mathbb F_q)$ over $\mathbb F_q$, we deduce that $A,B$ are similar.
\end{proof}

\subsection{Groups and completely regular semigroups}

If $S$ is a completely regular semigroup (that is, $s=s^{\omega+1}$ for all $s\in S$), then condition (2)  of Theorem~\ref{t:main.thm2} implies condition (1) of the theorem  and hence $\Bbbk$-linear conjugacy is the same as $\mathbb Q$-character equivalence plus  $\Bbbk$-character equivalence for completely regular semigroups; this applies, in particular, to groups.  Note that $\Bbbk$-character equivalence for groups was first described, in general, by Berman~\cite{Berman}. Let us spell out the characterization of $\Bbbk$-linear conjugacy explicitly for finite groups.

\begin{Thm}
Let $G$ be a finite group and $\Bbbk$ a field.  Then $g,h\in G$ are $\Bbbk$-linear conjugates if and only if they generate conjugate cyclic subgroups and are $\Bbbk$-character equivalent.
\end{Thm}

In positive characteristic, $\Bbbk$-character equivalence is different than $\Bbbk$-linear conjugacy for groups, as is easily seen by considering  a non-trivial $p$-group over a field $\Bbbk$ of characteristic $p$.  Indeed, all elements of a finite $p$-group $G$ are $\Bbbk$-character equivalent over a field of characteristic $p$ since the only irreducible representation of $G$ is the trivial representation.  But no non-trivial element is $\Bbbk$-linear conjugate to the identity.

\def\malce{\mathbin{\hbox{$\bigcirc$\rlap{\kern-7.75pt\raise0,50pt\hbox{${\tt
  m}$}}}}}\def\cprime{$'$} \def\cprime{$'$} \def\cprime{$'$} \def\cprime{$'$}
  \def\cprime{$'$} \def\cprime{$'$} \def\cprime{$'$} \def\cprime{$'$}
  \def\cprime{$'$} \def\cprime{$'$}

\end{document}